\documentclass[12pt]{amsart}
\usepackage{amssymb}
\usepackage[all]{xy}

\newtheorem{theorem}{Theorem}[section]
\newtheorem{definition}[theorem]{Definition}
\newtheorem{proposition}[theorem]{Proposition}
\newtheorem{conjecture}[theorem]{Conjecture}

\begin{document}

\title{Modelling questions for quantum permutations}

\author{Teodor Banica}
\address{T.B.: Department of Mathematics, University of Cergy-Pontoise, F-95000 Cergy-Pontoise, France. {\tt teo.banica@gmail.com}}

\author{Amaury Freslon}
\address{A.F.: Department of Mathematics, Paris-Sud University, F-91405 Orsay Cedex, France. {\tt amaury.freslon@math.u-psud.fr}}

\begin{abstract}
Given a quantum permutation group $G\subset S_N^+$, with orbits having the same size $K$, we construct a universal matrix model $\pi:C(G)\to M_K(C(X))$, having the property that the images of the standard coordinates $u_{ij}\in C(G)$ are projections of rank $\leq 1$. Our conjecture is that this model is inner faithful under suitable algebraic assumptions, and is in addition stationary under suitable analytic assumptions. We prove this conjecture for the classical groups, and for several key families of group duals.
\end{abstract}


\maketitle

\section*{Introduction}

The present paper is motivated by some questions in quantum algebra. Wang introduced in \cite{wa2} the free analogue $S_N^+$ of the quantum permutation group $S_N$. While many things are known about $S_N^+$, the analytic structure of the algebra $C(S_N^+)$ is still subject to investigation. One key open problem, slightly stronger than the Connes embedding property, is whether $C(S_N^+)$ has an inner faithful matrix model. See \cite{bfs}, \cite{bcv}, \cite{chi}.

A Grassmannian model approach to this question was proposed in \cite{bne}. The idea is that the magic bases of $\mathbb C^N$ form a real algebraic manifold $X_N$, and the problem is whether the corresponding representation $\pi_N:C(S_N^+)\to M_N(C(X_N))$ is inner faithful or not. In order to solve this question, some methods are available from \cite{bfs}, \cite{wa3}, but their application would require the construction of a measure on $X_N$. An idea here, proposed in \cite{bne}, is that of using the push-forward of the Haar measure on $U_N^N$, via a Sinkhorn type algorithm \cite{sin}. But the convergence of the algorithm is not established yet.

A perhaps simpler question, with many degrees of freedom, is that of looking first at the various quantum subgroups $G\subset S_N^+$. As explained in \cite{ban}, the matrix model construction is available in this setting, with the model space $X_G\subset X_N$ being obtained by imposing the abstract Tannakian conditions which define $G$. However, in the non-transitive case the model space collapses to the null space. We will fix here this issue with a new construction, the idea being to allow 0 entries in our magic basis when the orbits of $G$ are non-trivial. To be more precise, we will assume that $G$ is quasi-transitive, in the sense that its orbits have the same size $K\in\mathbb N$, with $K|N$, and we will construct a universal matrix model $\pi:C(G)\to M_K(C(X))$, having the property that the images of the standard coordinates $u_{ij}\in C(G)$ are projections of rank $\leq1$.

One important source of examples when trying to understand properties of compact quantum groups are duals of discrete groups. This is where our construction is interesting. Indeed, the only transitive group duals are cyclic groups, while there are plenty of quasi-transitive examples coming from free products of cyclic groups. We can therefore do computations and give explicit examples of inner faithful models in this enlarged setting.

Our conjecture is that the quasi-flat model is inner faithful under suitable uniformity assumptions on $G$, and is in addition stationary under suitable analytic assumptions on $G$. We will discuss this conjecture for the classical groups $G\subset S_N$, and we will investigate it as well for the group duals $\widehat{\Gamma}\subset S_N^+$. The general case, including that of $G=S_N^+$ itself, and of other transitive subgroups $G\subset S_N^+$, remains an open problem.

The paper is organized as follows: 1-2 contain preliminaries on quasi-transitive quantum groups, in 3-4 we construct the universal models and we formulate the conjectures, in 5-6 we perform some basic work on these conjectures, in the classical group and in the group dual cases, and in 7-8 we discuss in detail the group dual case.

\medskip

\noindent {\bf Acknowledgments.} We would like to thank A. Chirvasitu for useful discussions.

\section{Quantum permutations}

We are interested in the quantum analogues of the permutation groups $G\subset S_N$. In order to introduce these objects, let us recall that a magic unitary is a square matrix over a $C^*$-algebra, $u\in M_N(A)$,  whose entries are projections ($p^2=p^*=p$), summing up to $1$ on each row and each column. The following key definition is due to Wang \cite{wa2}:

\begin{definition}
$C(S_N^+)$ is the universal $C^*$-algebra generated by the entries of a $N\times N$ magic unitary matrix $u=(u_{ij})$, with the morphisms given by
$$\Delta(u_{ij})=\sum_ku_{ik}\otimes u_{kj}\quad,\quad\varepsilon(u_{ij})=\delta_{ij}\quad,\quad S(u_{ij})=u_{ji}$$
as comultiplication, counit and antipode. 
\end{definition}

This algebra satisfies Woronowicz' axioms in \cite{wo1}, \cite{wo2}, and the underlying space $S_N^+$ is therefore a compact quantum group, called quantum permutation group.

Observe that any magic unitary $v\in M_N(A)$ produces a representation $\pi:C(S_N^+)\to A$, given by $\pi(u_{ij})=v_{ij}$. In particular, we have a representation as follows:
$$\pi:C(S_N^+)\to C(S_N)\quad:\quad u_{ij}\to\chi\left(\sigma\in S_N\big|\sigma(j)=i\right)$$

The corresponding embedding $S_N\subset S_N^+$ is an isomorphism at $N=2,3$, but not at $N\geq4$, where $S_N^+$ is infinite. Moreover, it is known that we have $S_4^+\simeq SO_3^{-1}$, and that any $S_N^+$ with $N\geq4$ has the same fusion semiring as $SO_3$. See \cite{bco}, \cite{wa2}.

The orbit decomposition theory for the subgroups $G\subset S_N^+$ was developed in \cite{bic}. We present here an alternative approach, based on the following simple fact:

\begin{proposition}
Given a quantum group $G\subset S_N^+$, with standard coordinates denoted $u_{ij}\in C(G)$, the following defines an equivalence relation on $\{1,\ldots,N\}$:
$$i\sim j\ {\rm when}\ u_{ij}\neq0$$
In the classical case, $G\subset S_N$, this is the orbit equivalence coming from the action of $G$.
\end{proposition}

\begin{proof}
We first check the fact that we have indeed an equivalence relation:

(1) $i\sim i$ follows from $\varepsilon(u_{ij})=\delta_{ij}$, which gives $\varepsilon(u_{ii})=1$, and so $u_{ii}\neq0$, for any $i$.

(2) $i\sim j\implies j\sim i$ follows from $S(u_{ij})=u_{ji}$, which gives $u_{ij}\neq0\implies u_{ji}\neq0$.

(3) $i\sim j,j\sim k\implies i\sim k$ follows from $\Delta(u_{ik})=\sum_ju_{ij}\otimes u_{jk}$. Indeed, in this formula, the right-hand side is a sum of projections, so assuming $u_{ij}\neq0,u_{jk}\neq0$ for a certain index $j$, we have $u_{ij}\otimes u_{jk}>0$, and so $\Delta(u_{ik})>0$, which gives $u_{ik}\neq0$, as desired.

In the classical case now, $G\subset S_N$, the standard coordinates are the characteristic functions $u_{ij}=\chi(\sigma\in G|\sigma(j)=i)$. Thus the condition $u_{ij}\neq0$ is equivalent to the existence of an element $\sigma\in G$ such that $\sigma(j)=i$, and this means precisely that $i,j$ must be in the same orbit under the action of $G$, as claimed.
\end{proof}

Summarizing, we have a quantum analogue of the orbit decomposition from the classical case. It is convenient to introduce a few more related objects, as follows:

\begin{definition}
Associated to a quantum group $G\subset S_N^+$, producing as above the equivalence relation on $\{1,\ldots,N\}$ given by $i\sim j$ when $u_{ij}\neq0$, are as well:
\begin{enumerate}
\item The partition $\pi\in P(N)$ having as blocks the equivalence classes under $\sim$.

\item The binary matrix $\varepsilon\in M_N(0,1)$ given by $\varepsilon_{ij}=\delta_{u_{ij},0}$.
\end{enumerate}
\end{definition}

Observe that each of the objects $\sim,\pi,\varepsilon$ determines the other two ones. We will often assume, without mentioning it, that the orbits of $G\subset S_N^+$ come in increasing order, in the sense that the corresponding partition is as follows:
$$\pi=\{1,\ldots,K_1\},\ldots,\{K_1+\ldots+K_{M-1}+1,\ldots,K_1+\ldots+K_M\}$$

Indeed, at least for the questions that we are interested in here, we can always assume that it is so, simply by conjugating everything by a suitable permutation $\sigma\in S_N$.

In analogy with the classical case, we have as well the following notion:

\begin{definition}
We call $G\subset S_N^+$ transitive when $u_{ij}\neq 0$ for any $i,j$. Equivalently:
\begin{enumerate}
\item $\sim$ must be trivial, $i\sim j$ for any $i,j$.

\item $\pi$ must be the $1$-block partition.

\item $\varepsilon$ must be the all-$1$ matrix.
\end{enumerate}
\end{definition}

Let us discuss now the quantum analogue of the fact that given a subgroup $G\subset S_N$, with orbits of lenghts $K_1,\ldots,K_M$, we have an inclusion as follows:
$$G\subset S_{K_1}\times\ldots\times S_{K_M}$$

Given two quantum permutation groups $G\subset S_K^+,H\subset S_L^+$, with magic corepresentations denoted $u,v$, we can consider the algebra $A=C(G)*C(H)$, together with the magic matrix $w=diag(u,v)$. The pair $(A,w)$ satisfies Woronowicz's axioms, and we therefore obtain a quantum permutation group, denoted $G\,\hat{*}\,H\subset S_{K+L}^+$. See \cite{wa1}.

With this notion in hand, we have the following result:

\begin{proposition}
Given a quantum group $G\subset S_N^+$, with associated orbit decomposition partition $\pi\in P(N)$, having blocks of length $K_1,\ldots,K_M$, we have an inclusion
$$G\subset S_{K_1}^+\,\hat{*}\,\ldots\,\hat{*}\,S_{K_M}^+$$
where the product on the right is constructed with respect to the blocks of $\pi$. In the classical case, $G\subset S_N$, we obtain in this way the usual inclusion $G\subset S_{K_1}\times\ldots\times S_{K_M}$.
\end{proposition}

\begin{proof}
Since the standard coordinates $u_{ij}\in C(G)$ satisfy $u_{ij}=0$ for $i\not\sim j$, the algebra $C(G)$ appears as quotient of the following algebra:
\begin{eqnarray*}
C(S_N^+)\Big/\left<u_{ij}=0,\forall i\not\sim j\right>
&=&C(S_{K_1}^+)*\ldots*C(S_{K_M}^+)\\
&=&C(S_{K_1}^+\,\hat{*}\,\ldots\,\hat{*}\,S_{K_M}^+)
\end{eqnarray*}

Thus, we have an inclusion of quantum groups, as in the statement. Finally, observe that the classical version of the quantum group $S_{K_1}^+\,\hat{*}\,\ldots\,\hat{*}\,S_{K_M}^+$ is given by:
\begin{eqnarray*}
(S_{K_1}^+\,\hat{*}\,\ldots\,\hat{*}\,S_{K_M}^+)_{class}
&=&(S_{K_1}\times\ldots\times S_{K_M})_{class}\\
&=&S_{K_1}\times\ldots\times S_{K_M}
\end{eqnarray*}

Thus in the classical case we obtain $G\subset S_{K_1}\times\ldots\times S_{K_M}$, as claimed.
\end{proof}

Let us discuss now what happens in the group dual case, where the situation is non-trivial. Following the work of Bichon in \cite{bic}, we have the following result:

\begin{proposition}
Given a decomposition $N=K_1+\ldots+K_M$, and a quotient group $\mathbb Z_{K_1}*\ldots*\mathbb Z_{K_M}\to\Gamma$, we have an embedding, as follows:
$$\widehat{\Gamma}\subset\mathbb Z_{K_1}\,\hat{*}\,\ldots\,\hat{*}\,\mathbb Z_{K_M}\subset S_{K_1}^+\,\hat{*}\,\ldots\,\hat{*}\,S_{K_M}^+\subset S_N^+$$
Moreover, modulo the action of $S_N\times S_N$ on the magic unitaries, obtained by permuting the rows and columns, we obtain in this way all the group dual subgroups $\widehat{\Gamma}\subset S_N^+$.
\end{proposition}

\begin{proof}
Given a quotient group $\Gamma$ as in the statement, by composing a number of standard embeddings and identifications, we obtain indeed an embedding, as follows:
\begin{eqnarray*}
\widehat{\Gamma}
&\subset&\widehat{\mathbb Z_{K_1}*\ldots*\mathbb Z_{K_M}}
=\widehat{\mathbb Z}_{K_1}\,\hat{*}\,\ldots\,\hat{*}\,\widehat{\mathbb Z}_{K_M}
\simeq\mathbb Z_{K_1}\,\hat{*}\,\ldots\,\hat{*}\,\mathbb Z_{K_M}\\
&\subset&S_{K_1}\hat{*}\,\ldots\,\hat{*}\,S_{K_M}
\subset S_{K_1}^+\hat{*}\,\ldots\,\hat{*}\,S_{K_M}^+
\subset S_{K_1+\ldots+K_M}^+
\end{eqnarray*}

Regarding now the last assertion, this basically follows by letting  $N=K_1+\ldots+K_M$ be the decomposition coming from the orbit structure of $\widehat{\Gamma}\subset S_N^+$. See \cite{bic}.
\end{proof}

Let us now consider the case where the decomposition $N=K_1+\ldots+K_M$ is ``minimal'', in the sense that the quotient map $\mathbb Z_{K_1}*\ldots*\mathbb Z_{K_M}\to\Gamma$ is faithful on each $\mathbb Z_{K_i}$. With this assumption made, we have:

\begin{theorem}
Assume that $\widehat{\Gamma}\subset S_N^+$ comes from a quotient group $\mathbb Z_{K_1}*\ldots*\mathbb Z_{K_M}\to\Gamma$ with $K_1+\ldots+K_M=N$, such that the quotient map is faithful on each $\mathbb Z_{K_i}$.
\begin{enumerate}
\item The associated orbit decomposition is $N=K_1+\ldots+K_M$.

\item The inclusions $\widehat{\Gamma}\subset S_{K_1}^+\,\hat{*}\,\ldots\,\hat{*}\,S_{K_M}^+$ from Propositions 1.5 and 1.6 coincide.
\end{enumerate}
\end{theorem}

\begin{proof}
We recall from Proposition 1.6 that the subgroup $\widehat{\Gamma}\subset S_N^+$ appears as follows:
$$\widehat{\Gamma}\subset\mathbb Z_{K_1}\,\hat{*}\,\ldots\,\hat{*}\,\mathbb Z_{K_M}\subset S_{K_1}^+\,\hat{*}\,\ldots\,\hat{*}\,S_{K_M}^+\subset S_N^+$$

(1) By construction of $\widehat{\Gamma}\subset S_N^+$, the orbit decomposition for this quantum group must appear via a refinement of $N=K_1+\ldots+K_M$. On the other hand, since $\mathbb Z_{K_1}*\ldots*\mathbb Z_{K_M}\to\Gamma$ is faithful on each $\mathbb Z_{K_i}$, the elements $(K_1+\ldots+K_{i-1})+1, \ldots,(K_1+\ldots+K_{i-1})+K_i$ must belong to the same orbit under the action of $\widehat{\Gamma}$, and we are done.

(2) This is just an observation, which is clear from (1) above.
\end{proof}

For more on the group duals $\widehat{\Gamma}\subset S_N^+$, we refer to \cite{bic}. We will come back later to these quantum groups, under the extra assumption $K_1=\ldots=K_M$.

\section{Quasi-transitivity}

We discuss now an extension of the notion of transitivity, that we call quasi-transitivity. We will see later on that the universal flat matrix model construction from \cite{ban}, \cite{bne}, which works well in the transitive case, adapts to the quasi-transitive case. 

In terms of the objects $\sim,\pi,\varepsilon$ introduced above, we have:

\begin{definition}
A quantum permutation group $G\subset S_N^+$ is called quasi-transitive when all its orbits have the same size. Equivalently:
\begin{enumerate}
\item $\sim$ has equivalence classes of same size.

\item $\pi$ has all the blocks of equal length.

\item $\varepsilon$ is block-diagonal with blocks the flat matrix of size $K$.
\end{enumerate}
\end{definition}

As a first example, if $G$ is transitive then it is quasi-transitive. In general now, if we denote by $K\in\mathbb N$ the common size of the blocks, and by $M\in\mathbb N$ their multiplicity, then we must have $N=KM$. We have the following result:

\begin{proposition}
Assuming that $G\subset S_N^+$ is quasi-transitive, we must have
$$G\subset\underbrace{S_K^+\,\hat{*}\,\ldots\,\hat{*}\,S_K^+}_{M\ terms}$$
where $K\in\mathbb N$ is the common size of the orbits, and $M\in\mathbb N$ is their number.
\end{proposition}

\begin{proof}
This simply follows from Proposition 1.5 above, because, with the notations there, in the quasi-transitive case we must have $K_1=\ldots=K_M=K$. Observe that in the classical case, we obtain in this way the usual embedding $G\subset\underbrace{S_K\times\ldots\times S_K}_{M\ terms}$.
\end{proof}

Let us discuss now the examples. Assume that $G\subset S_K^+\,\hat{*}\,\ldots\,\hat{*}\,S_K^+$. If $u,v$ are the fundamental corepresentations of $C(S_N^+),C(S_K^+)$, consider the quotient map $\pi_i:C(S_N^+)\to C(S_K^+)$ constructed as follows: 
$$u\to diag(1_K,\ldots,1_K,\underbrace{v}_{i-th\ term},1_K,\ldots,1_K)$$

We can then set $C(G_{i}) = \pi_{i}(C(G))$, and we have the following result:

\begin{proposition}
If $G_{i}$ is transitive for all $i$, then $G$ is quasi-transitive.
\end{proposition}

\begin{proof}
We have embeddings as follows:
$$G_1\times\ldots\times G_M\subset G\subset\underbrace{S_K^+\,\hat{*}\,\ldots\,\hat{*}\,S_K^+}_{M\ terms}$$

It follows that the size of any orbit of $G$ is at least $K$ (it contains $G_1\times\ldots\times G_M$) and at most $K$ (it is contained in $S_K^+\,\hat{*}\,\ldots\,\hat{*}\,S_K^+$). Thus, $G$ is quasi-transitive.
\end{proof}

We call the quasi-transitive subgroups appearing as above ``of product type''. Observe that there are quasi-transitive groups which are not of product type, as for instance the group $G=S_2\subset S_2\times S_2\subset S_4$ obtained by using the embedding $\sigma\to(\sigma,\sigma)$. Indeed, the quasi-transitivity is clear, say by letting $G$ act on the vertices of a square. On the other hand, since we have $G_1=G_2=\{1\}$, this group is not of product type.

In general, we can construct examples by using various product operations:

\begin{proposition}
Given transitive subgroups $G_1,\ldots,G_M\subset S_K^+$, the following constructions produce quasi-transitive subgroups $G\subset\underbrace{S_K^+\,\hat{*}\,\ldots\,\hat{*}\,S_K^+}_{M\ terms}$, of product type:
\begin{enumerate}
\item The usual product: $G=G_1\times\ldots\times G_M$.

\item The dual free product: $G=G_1\,\hat{*}\,\ldots\,\hat{*}\, G_M$.
\end{enumerate}
\end{proposition}

\begin{proof}
All these assertions are clear from definitions, because in each case, the quantum groups $G_i\subset S_K^+$ constructed in Proposition 2.3 are those in the statement.
\end{proof}

In the group dual case, we have the following result:

\begin{proposition}
The group duals $\widehat{\Gamma}\subset\underbrace{S_K^+\,\hat{*}\,\ldots\,\hat{*}\,S_K^+}_{M\ terms}$ which are of product type are precisely those appearing from intermediate groups of the following type:
$$\underbrace{\mathbb Z_K*\ldots*\mathbb Z_K}_{M\ terms}\to\Gamma\to\underbrace{\mathbb Z_K\times\ldots\times\mathbb Z_K}_{M\ terms}$$
\end{proposition}

\begin{proof}
It is clear that any intermediate quotient $\Gamma$ as in the statement produces a quantum permutation group $\widehat{\Gamma}\subset S_N^+$ which is of product type. Conversely, given a group dual $\widehat{\Gamma}\subset S_N^+$, coming from a quotient group $\mathbb Z_K^{* M}\to\Gamma$, the subgroups $G_i\subset\widehat{\Gamma}$ constructed in Proposition 2.3 must be group duals as well, $G_i=\widehat{\Gamma}_i$, for certain quotient groups $\Gamma\to\Gamma_i$. Now if $\widehat{\Gamma}$ is of product type, $\widehat{\Gamma}_i\subset S_K^+$ must be transitive, and hence equal to $\widehat{\mathbb Z}_K$. We then conclude that we have $\widehat{\mathbb Z_K^M}\subset\widehat{\Gamma}$, and so $\Gamma\to\mathbb Z_K^M$, as in the proof of Proposition 2.3.
\end{proof}

In order to give now some other classes of examples, we use the notion of normality for compact quantum groups, from \cite{cdp}, \cite{wa3}. This notion is introduced as follows:

\begin{definition}
Given a quantum subgroup $H\subset G$, coming from a quotient map $\pi:C(G)\to C(H)$, the following are equivalent:
\begin{enumerate}
\item $A=\{a\in C(G)|(id\otimes\pi)\Delta(a)=a\otimes1\}$ satisfies $\Delta(A)\subset A\otimes A$.

\item $B=\{a\in C(G)|(\pi\otimes id)\Delta(a)=1\otimes a\}$ satisfies $\Delta(B)\subset B\otimes B$.

\item We have $A=B$, as subalgebras of $C(G)$.
\end{enumerate}
If these conditions are satisfied, we say that $H\subset G$ is a normal subgroup.
\end{definition}

As explained in \cite{cdp}, in the classical case we obtain the usual normality notion for the subgroups. Also, in the group dual case the normality of any subgroup, which must be a group dual subgroup, is automatic. Now with this notion in hand, we have:

\begin{theorem}
Assuming that $G\subset S_N^+$ is transitive, and that $H\subset G$ is normal, $H\subset S_N^+$ follows to be quasi-transitive.
\end{theorem}

\begin{proof}
Consider the quotient map $\pi:C(G)\to C(H)$, as in Definition 2.6, given at the level of standard coordinates by $u_{ij}\mapsto v_{ij}$. Consider two orbits $O_1,O_2$ of $H$ and set:
$$x_i=\sum_{j\in O_1}u_{ij}\quad,\quad y_i=\sum_{j\in O_2}u_{ij}$$

These two elements are orthogonal projections in $C(G)$ and they are nonzero, because they are sums of nonzero projections by transitivity of $G$. We have:
$$(id\otimes\pi)\Delta(x_i)
=\sum_k\sum_{j\in O_1}u_{ik}\otimes v_{kj}
=\sum_{k\in O_1}\sum_{j\in O_1}u_{ik}\otimes v_{kj}
=\sum_{k\in O_1}u_{ik}\otimes 1
=x_i\otimes 1$$

Thus by normality of $H$ we have $(\pi\otimes id)\Delta(x_i)=1\otimes x_i$. On the other hand, assuming that we have $i\in O_2$, we obtain:
$$(\pi\otimes id)\Delta(x_i)=\sum_k\sum_{j\in O_1}v_{ik}\otimes u_{kj}=\sum_{k\in O_2}v_{ik}\otimes x_k$$

Multiplying this by $v_{ik}\otimes 1$ with $k\in O_2$ yields $v_{ik}\otimes x_k=v_{ik}\otimes x_i$, that is to say $x_k=x_i$. In other words, $x_i$ only depends on the orbit of $i$. The same is of course true for $y_i$.

By using this observation, we can compute the following element:
$$z=\sum_{k\in O_2}\sum_{j\in O_1}u_{kj}=\sum_{k\in O_2}x_k=\vert O_2\vert x_i$$

On the other hand, by applying the antipode, we have as well:
$$S(z)=\sum_{k\in O_2}\sum_{j\in O_1}u_{jk}=\sum_{j\in O_1}y_j=\vert O_1\vert y_j$$

We therefore obtain the following formula:
$$S(x_i)=\frac{\vert O_1\vert}{\vert O_2\vert}y_j$$

Now since both $x_i$ and $y_j$ have norm one, we conclude that the two orbits have the same size, and this finishes the proof.
\end{proof}

Some additional interesting transitivity questions appear in the graph context. See \cite{cha}.

\section{Matrix models}

Given a quantum permutation group $G\subset S_N^+$, we will be interested in what follows in the matrix models of type $\pi:C(G)\to M_K(C(X))$, with $X$ being a compact space. There are many examples of such models, and the ``simplest'' ones are as follows:

\begin{definition}
Given a subgroup $G\subset S_N^+$, a random matrix model of type
$$\pi:C(G)\to M_K(C(X))$$
is called quasi-flat when the fibers $P_{ij}^x=\pi(u_{ij})(x)$ all have rank $\leq1$. 
\end{definition}

As a first observation, the functions $x\mapsto r_{ij}^x=rank(P_{ij}^x)$ are locally constant over $X$, so they are constant over the connected components of $X$. Thus, when $X$ is connected, our assumption is that we have $r_{ij}^x=r_{ij}\in\{0,1\}$, for any $x\in X$, and any $i,j$.

Observe that in the case $K=N$ these questions disappear, because we must have $r_{ij}^x=1$ for any $i,j$, and any $x\in X$. In this case the model is called flat. See \cite{bne}.

\begin{proposition}
Assume that we have a quasi-flat model $\pi:C(G)\to M_K(C(X))$, mapping $u_{ij}\mapsto P_{ij}$, and consider the matrix $r_{ij}=rank(P_{ij})$.
\begin{enumerate}
\item $r$ is bistochastic, with sums $K$.

\item We have $r_{ij}\leq\varepsilon_{ij}$, for any $i,j$.

\item If $G$ is quasi-transitive, with orbits of size $K$, then $r_{ij}=\varepsilon_{ij}$ for any $i,j$.

\item If $\pi$ is assumed to be flat, then $G$ must be transitive.
\end{enumerate}
\end{proposition}

\begin{proof}
These results are all elementary, the proof being as follows:

(1) This is clear from the fact that each $P^x=(P_{ij}^x)$ is bistochastic, with sums $1$.

(2) This simply comes from $u_{ij}=0\implies P_{ij}=0$.

(3) The matrices $r=(r_{ij})$ and $\varepsilon=(\varepsilon_{ij})$ are both bistochastic, with sums $K$, and they satisfy $r_{ij}\leq\varepsilon_{ij}$, for any $i,j$. Thus, these matrices must be equal, as stated.

(4) This is clear, because $rank(P_{ij})=1$ implies $u_{ij}\neq0$, for any $i,j$.
\end{proof}

In order to construct now universal quasi-flat models, we use the following standard result from \cite{ban}, which is a reformulation of Woronowicz's Tannakian duality \cite{wo2}:

\begin{proposition}
Given an inclusion $G\subset S_N^+$, with the corresponding fundamental corepresentations denoted $w\mapsto u$, we have the following formula:
$$C(G)=C(S_N^+)\Big/\Big(T\in Hom(w^{\otimes k},w^{\otimes l}),\forall k,l\in\mathbb N,\forall T\in Hom(u^{\otimes k},u^{\otimes l})\Big)$$
with the Hom-spaces at left being taken in a formal sense.
\end{proposition}

\begin{proof}
We recall that for a Hopf algebra corepresentation $v=(v_{ij})$, the intertwining condition $T\in Hom(v^{\otimes k},v^{\otimes l})$ means by definition that we have $Tv^{\otimes k}=v^{\otimes l}T$, the tensor powers of $v=(v_{ij})$ being the corepresentations $v^{\otimes r}=(v_{i_1\ldots i_r,j_1\ldots j_r})$. 

We can formally use these notions for any square matrix over any $C^*$-algebra, and in particular, for the fundamental corepresentation of $C(S_N^+)$. Thus, the collection of relations $T\in Hom(w^{\otimes k},w^{\otimes l})$, one for each choice of an intertwiner $T\in Hom(u^{\otimes k},u^{\otimes l})$, produce an ideal of $C(S_N^+)$, and the algebra in the statement is well-defined.

This latter algebra is isomorphic to $C(G)$, due to Woronowicz's Tannakian results in \cite{wo2}. For a short, recent proof here, using basic Hopf algebra theory, see \cite{mal}.
\end{proof}

Now back to our modelling questions, it is convenient to identify the rank one projections in $M_N(\mathbb C)$ with the elements of the complex projective space $P^{N-1}_\mathbb C$. 

We first have the following observation, which goes back to \cite{bne}:

\begin{proposition}
The algebra $C(S_N^+)$ has a universal flat model, given by 
$$\pi_N:C(S_N^+)\to M_N(C(X_N))\quad,\quad\pi_N(u_{ij})=[P\mapsto P_{ij}]$$
where $X_N$ is the set of matrices $P\in M_N(P^{N-1}_\mathbb C)$ which are bistochastic with sums $1$.
\end{proposition}

\begin{proof}
This is clear from definitions, because any flat model $C(S_N^+)\to M_N(\mathbb C)$ must map the magic corepresentation $u=(u_{ij})$ into a matrix $P=(P_{ij})$ belonging to $X_N$.
\end{proof}

Regarding now the general quasi-transitive case, we have here:

\begin{theorem}
Given a quasi-transitive subgroup $G\subset S_N^+$, with orbits of size $K$, we have a universal quasi-flat model $\pi:C(G)\to M_K(C(X))$, constructed as follows:
\begin{enumerate}
\item For $G=\underbrace{S_K^+\,\hat{*}\,\ldots\,\hat{*}\,S_K^+}_{M\ terms}$ with $N=KM$, the model space is $X_{N,K}=\underbrace{X_K\times\ldots\times X_K}_{M\ terms}$, and with $u=diag(u^1,\ldots,u^M)$ the map is $\pi_{N,K}(u^r_{ij})=[(P^1,\ldots,P^M)\mapsto P^r_{ij}]$. 

\item In general, the model space is the submanifold $X_G\subset X_{N,K}$ obtained via the Tannakian relations defining $G$.
\end{enumerate}
\end{theorem}

\begin{proof}
This result is known since \cite{ban}, \cite{bne} in the flat case, the idea being to use Proposition 3.3 and Proposition 3.4. In general, the proof is similar:

(1) This follows from Proposition 3.4, by using Proposition 3.2 (3) above, which tells us that the 0 entries of the model must appear exactly where $u=(u_{ij})$ has 0 entries.

(2) Assume that $G\subset S_N^+$ is quasi-transitive, with orbits of size $K$. We have then an inclusion $G\subset\underbrace{S_K^+\,\hat{*}\,\ldots\,\hat{*}\,S_K^+}_{M\ terms}$, and in order to construct the universal quasi-flat model for $C(G)$, we need a universal solution to the following factorization problem:
$$\begin{matrix}
C(\underbrace{S_K^+\,\hat{*}\,\ldots\,\hat{*}\,S_K^+}_{M\ terms})&\to&M_K(C(X_{N,K}))\\
\\
\downarrow&&\downarrow\\
\\
C(G)&\to&M_K(C(X_G))
\end{matrix}$$

But, the solution to this latter question is given by the following construction, with the Hom-spaces at left being taken as usual in a formal sense:
$$C(X_G)=C(X_{N,K})\Big/\Big(T\in Hom(P^{\otimes k},P^{\otimes l}),\forall k,l\in\mathbb N,\forall T\in Hom(u^{\otimes k},u^{\otimes l})\Big)$$

With this result in hand, the Gelfand spectrum of the algebra on the left is then an algebraic submanifold $X_G\subset X_{N,K}$, having the desired universality property.
\end{proof}

Observe that talking about quasi-flat models for quantum groups which are not necessarily quasi-transitive perfectly makes sense. The universal model spaces can be constructed as above, and this was in fact already discussed in \cite{ban}, but no one guarantees that in the non-quasi-transitive case, the model spaces are non-empty. So, we prefer to restrict the attention to the quasi-transitive case, and state Theorem 3.5 as it is.

\section{Inner faithfulness}

We formulate in what follows a number of conjectures. We first review the notions of inner faithfulness and stationarity, from \cite{ban}, \cite{bbi}, \cite{bfs}. Following \cite{bbi}, we first have:

\begin{definition}
Let $\pi:C(G)\to M_K(C(X))$ be a matrix model. 
\begin{enumerate}
\item The Hopf image of $\pi$ is the smallest quotient Hopf $C^*$-algebra $C(G)\to C(H)$ producing a factorization of type $\pi:C(G)\to C(H)\to M_K(C(X))$.

\item When the inclusion $H\subset G$ is an isomorphism, i.e. when there is no non-trivial factorization as above, we say that $\pi$ is inner faithful.
\end{enumerate}
\end{definition}

Observe that when $G=\widehat{\Gamma}$ is a group dual, $\pi$ must come from a group representation $\rho:\Gamma\to C(X,U_K)$, and the above factorization is the one obtained by taking the image, $\rho:\Gamma\to\Gamma'\subset C(X,U_K)$. Thus $\pi$ is inner faithful when $\Gamma\subset C(X,U_K)$.

Also, given a compact group $G$, and elements $g_1,\ldots,g_K\in G$, we have a representation $\pi:C(G)\to\mathbb C^K$, given by $f\to(f(g_1),\ldots,f(g_K))$. The minimal factorization of $\pi$ is then via $C(G')$, with $G'=\overline{<g_1,\ldots,g_K>}$, and $\pi$ is inner faithful when $G=G'$.

In practice, $X$ is often a compact Lie group, or a compact homogeneous space, or a more general compact probability space. And here, we have the following result:

\begin{proposition}
Given an inner faithful model $\pi:C(G)\to M_K(C(X))$, with $X$ being assumed to be a compact probability space, we have
$$\int_G=\lim_{k\to\infty}\frac{1}{k}\sum_{r=1}^k\int_G^r$$
where $\int_G^r=(\varphi\circ\pi)^{*r}$, with $\varphi=tr\otimes\int_X$ being the random matrix trace.
\end{proposition}

\begin{proof}
This was proved in \cite{bfs} in the case $X=\{\cdot\}$, using idempotent state theory from \cite{fsk}. The general case was recently established in \cite{wa4}.
\end{proof}

The above result can be used as a criterion for detecting the inner faithfulness. To be more precise, $\pi$ is inner faithful precisely when the above formula holds. See \cite{bfs}.

Following \cite{ban}, we call a matrix model stationary when the Ces\`aro limiting convergence in Proposition 4.2 is stationary. In other words, we have the following definition:

\begin{definition}
A matrix model $\pi:C(G)\to M_K(C(X))$, with $X$ assumed to be a compact probability space, is called stationary when: 
$$\int_G=\left(tr\otimes\int_X\right)\pi$$
In the general case, where $X$ is only assumed to be a compact space, $\pi$ will be called stationary if it is stationary with respect to some probability measure on $X$.
\end{definition}

There are many interesting examples of such models, see \cite{ban}. However, the stationarity condition is a very strong assumption, and we have the following result, from \cite{ban}:

\begin{proposition}
Let $\pi:C(G)\to M_K(C(X))$ be a stationary model. Then,
\begin{enumerate}
\item $\pi$ is faithful.

\item $C(G)$ is a type I C*-algebra, hence the discrete dual $\Gamma=\widehat{G}$ is amenable.
\end{enumerate}
\end{proposition}

\begin{proof}
We use the basic theory of amenability for discrete quantum groups, from \cite{ntu}. Assuming that $\pi$ is stationary, for any $x\in C(G)$ we have:
$$\int_Gx=0\implies\left(tr\otimes\int_X\right)\pi(x)=0$$

In particular, with $x=yy^*$, and by using the fact that $\int_X$ is by definition faithful, $X$ being a compact probability space, we obtain that for any $y\in C(G)$ we have: 
$$\int_Gyy^*=0\implies\pi(yy^*)=0\implies\pi(y)=0$$

Now since the elements satisfying $\int_Gyy^*=0$ are precisely those in the kernel of the quotient map $\lambda:C(G)\to C(G)_{red}$, we obtain a factorization of $\pi$, as follows:
$$\pi:C(G)\to C(G)_{red}\to M_K(C(X))$$

Our claim now is that the map on the right, say $\rho$, is an inclusion. Indeed, let $x\in\ker(\rho)$, and let us pick a lift $y\in C(G)$ of this element $x\in C(G)_{red}$. We have then:
$$\rho(x)=0\implies\pi(y)=0\implies\pi(yy^*)=0\implies\int_Gyy^*=0\implies\lambda(y)=0\implies x=0$$

Thus we have an inclusion $C(G)_{red}\subset M_K(C(X))$, and so $\pi$ factorizes as follows:
$$\pi:C(G)\to C(G)_{red}\subset M_K(C(X))$$

Now since $C(G)_{red}$ must be of type I, and therefore nuclear, the quantum group $G$ must be co-amenable, and so $\pi$ must be faithful, and we are done with both (1,2).
\end{proof}

We refer to \cite{ban}, \cite{bbi}, \cite{bfs}, \cite{chi}, \cite{fsk}, \cite{wa4} for more theory and examples, of algebraic and analytic nature, regarding the notions of inner faithfulness and stationarity.

We recall from \cite{wo1} that any finitely generated group $\Gamma=<g_1,\ldots,g_N>$ produces a closed subgroup $\widehat{\Gamma}\subset U_N^+$, with fundamental corepresentation $u_{ij}=\delta_{ij}g_i$, and that the group dual subgroups $\widehat{\Gamma}\subset U_N^+$ all appear in this way, modulo a conjugation of the fundamental corepresentation $u=(u_{ij})$ by a unitary $U\in U_N$. Here is now a first result about stationarity, which is essentially a reformulation of Thoma's theorem \cite{tho}:

\begin{theorem}[Thoma]
Given a group dual $G=\widehat{\Gamma}\subset U_N^+$, the following are equivalent:
\begin{enumerate}
\item $C(G)$ is of type I.

\item $C(G)$ has a stationary model.

\item $C(G)$ has a stationary model, over an homogeneous space.

\item $\Gamma$ is virtually abelian.
\end{enumerate}
\end{theorem}

\begin{proof}
Here $(1)\implies(4)$ is Thoma's theorem \cite{tho} and $(3)\implies(2)\implies(1)$ are trivial implications. We therefore only have to prove $(4)\implies(3)$.

Let $\Lambda < \Gamma$ be an abelian subgroup of finite index and let $K = [\Lambda : \Gamma]$. We define a matrix model $\pi : C^*(\Gamma)\to M_K(C(\widehat{\Lambda}))$ by:
$$\pi(\gamma)(\chi) = \mathrm{Ind}_\Lambda^\Gamma(\chi)(\gamma)$$

To see that this model is faithful, take $\gamma\in \Gamma$ and recall that the character $\psi$ of $\mathrm{Ind}_\Lambda^\Gamma(\chi)$ is given by:
$$\psi(\gamma) = Tr(\mathrm{Ind}_\Lambda^\Gamma(\chi)) = \sum_{x\in \Gamma/\Lambda}\delta_{x^{-1}\gamma x\in \Lambda}\chi(x^{-1}\gamma x)$$

Here $Tr$ is the usual (non-normalized) trace on $M_{K}(\mathbb C)$. Thus:
$$\left(tr\otimes \int_{\widehat{\Lambda}}\right)\pi(\gamma) = \frac{1}{K}\sum_{x\in \Gamma/\Lambda}\delta_{x^{-1}\gamma x\in \Lambda}\int_{\widehat{\Lambda}}\chi(x^{-1}\gamma x)d\chi$$

Since the integral over all characters is the indicator function of the trivial element, the expression above equals $\delta_{\gamma, e}$ and the model is stationary.
\end{proof}

Let us formulate now our main two conjectures, regarding the universal quasi-flat models for the quantum permutation groups. We first have:

\begin{conjecture}
Assuming that $G\subset S_N^+$ is quasi-transitive, with orbits of size $K$, and that $\Gamma=\widehat{G}$ satisfies a suitable ``virtual abelianity'' condition, we have:
\begin{enumerate}
\item The universal quasi-flat model space $X_G$ is an homogeneous space.

\item The corresponding model $\pi:C(G)\to M_K(C(X_G))$ is stationary. 
\end{enumerate}
\end{conjecture}

The evidence here comes from Thoma's theorem, in its conjectural stronger form presented above, as well from a number of explicit verifications, to be performed below, and notably from a verification in the case where $G$ is classical.

We do not know yet what the ``virtual abelianity'' condition should mean. When $\Gamma$ is classical, as in Thoma's theorem, this condition states that we must have an abelian subgroup $\Lambda <\Gamma$ such that $F=\Gamma/\Lambda$ is finite.

Regarding now the notion of inner faithfulness, we have here:

\begin{conjecture}
Assuming that $G\subset S_N^+$ is quasi-transitive, and satisfies in addition a  suitable ``uniformity'' condition, the universal quasi-flat model is inner faithful.
\end{conjecture}

We do not know yet what the precise ``uniformity'' condition should be. We believe that all this is related to the notion of easiness \cite{fre}, \cite{rwe}, and this will be confirmed by some of the verifications performed below, but in general, we have no results.

Finally, let us mention that in the transitive case, the very first question here concerns $G=S_N^+$ itself, and the problem here is difficult, and open. Indeed, assuming that the conjecture holds, $S_N^+$ would follow to be inner linear (in a parametric sense) and we would therefore obtain that $L^\infty(S_N^+)$ has the Connes embedding property. Thus, we will have here a solution to an old open problem. For some comments here, see \cite{bfs}, \cite{bne}, \cite{bcv}.

\section{The classical case}

In this section we discuss the classical case, $G\subset S_N$. Our question is as follows: assuming that $G$ is quasi-transitive, with orbits of size $K$, when do we have a stationary model $\pi:C(G)\to M_K(C(X))$, for some compact probability space $X$? 

We will use the following notion:

\begin{definition}
A ``sparse Latin square'' is a square matrix $L\in M_N(*,1,\ldots,K)$ having the property that each of its rows and columns consists of a permutation of the numbers $1,\ldots,K$, completed with $*$ entries.
\end{definition}

In the case $K=N$, where there are no $*$ symbols, we recover the usual Latin squares. In general, however, the combinatorics of these matrices seems to be more complicated than that of the usual Latin squares. Here are a few examples of such matrices:
$$\begin{pmatrix}1&2&*\\2&*&1\\ *&1&2\end{pmatrix}\quad,\quad
\begin{pmatrix}1&2&*&*\\ 2&*&1&*\\ *&1&*&2\\ *&*&2&1\end{pmatrix}\quad,\quad
\begin{pmatrix}1&2&*&*\\ 2&1&*&*\\ *&*&1&2\\ *&*&2&1\end{pmatrix}\quad,\quad
\begin{pmatrix}1&2&*&*&*\\ 2&*&1&*&*\\ *&1&2&*&*\\ *&*&*&1&2\\ *&*&*&2&1
\end{pmatrix}$$

With this notion in hand, the result that we need is as follows:

\begin{proposition}
The quasi-flat representations $\pi:C(S_N)\to M_K(\mathbb C)$ appear as
$$u_{ij}\mapsto P_{L_{ij}}$$ 
where $P_1,\ldots,P_K\in M_K(\mathbb C)$ are rank $1$ projections, summing up to $1$, and where $L\in M_N(*,1,\ldots,K)$ is a sparse Latin square, with the convention $P_*=0$. 
\end{proposition}

\begin{proof}
Assuming that $\pi:C(S_N)\to M_K(\mathbb C)$ is quasi-flat, the elements $P_{ij}=\pi(u_{ij})$ are projections of rank $\leq1$, which pairwise commute, and form a magic unitary. 

Let $P_1,\ldots,P_K\in M_K(\mathbb C)$ be the rank one projections appearing in the first row of $P=(P_{ij})$. Since these projections form a partition of unity with rank one projections, any rank one projection $Q\in M_K(\mathbb C)$ commuting with all of them satisfies $Q\in\{P_1,\ldots,P_K\}$. In particular we have $P_{ij}\in\{P_1,\ldots,P_K\}$ for any $i,j$ such that $P_{ij}\neq 0$. Thus we can write $u_{ij}\mapsto P_{L_{ij}}$, for a certain matrix $L\in M_N(*,1,\ldots,K)$, with the convention $P_*=0$.

In order to finish, the remark is that $u_{ij}\mapsto P_{L_{ij}}$ defines a representation $\pi:C(S_N)\to M_K(\mathbb C)$ precisely when the matrix $P=(P_{L_{ij}})_{ij}$ is magic. But this condition tells us precisely that $L$ must be a sparse Latin square, in the sense of Definition 5.1.
\end{proof}

Our task now is to compute the associated Hopf image. We have here:

\begin{proposition}
Given a sparse Latin square $L\in M_N(*,1,\ldots,K)$, consider the permutations $\sigma_1,\ldots,\sigma_K\in S_N$ given by:
$$\sigma_x(j)=i\iff L_{ij}=x$$
The Hopf image associated to a representation $\pi:C(S_N)\to M_K(\mathbb C)$, $u_{ij}\mapsto P_{L_{ij}}$ as above is then the algebra $C(G_L)$, where $G_L=<\sigma_1,\ldots,\sigma_K>\subset S_N$.
\end{proposition}

\begin{proof}
We use a method from \cite{bbs}. The image of $\pi$ being generated by $P_1,\ldots,P_K$, we have an isomorphism of algebras $\alpha:Im(\pi)\simeq C(1,\ldots,K)$ given by $P_i\mapsto\delta_i$. Consider the following diagram:
$$\xymatrix@R=10mm@C=10mm{
C(S_N)\ar[r]^\pi\ar[dr]_\varphi&Im(\pi)\ar[r]\ar[d]^\alpha&M_K(\mathbb C)\\
&C(1,\ldots,K)
}$$

Here the map on the right is the canonical inclusion and $\varphi=\alpha\pi$. Since the Hopf image of $\pi$ coincides with the one of $\varphi$, it is enough to compute the latter. We know that $\varphi$ is given by $\varphi(u_{ij})=\delta_{L_{ij}}$, with the convention $\delta_*=0$. By Gelfand duality, $\varphi$ must come from a certain map $\sigma:\{1,\ldots,K\}\to S_N$, via the transposition formula $\varphi(f)(x)=f(\sigma_x)$. With the choice $f=u_{ij}$, we obtain:
$$\delta_{L_{ij}}(x)=u_{ij}(\sigma_x)$$

Now observe that these two quantities are by definition given by:
$$\delta_{L_{ij}}(x)=\begin{cases}
1&{\rm if}\ L_{ij}=x\\
0&{\rm otherwise}
\end{cases}\qquad,\qquad 
u_{ij}(\sigma_x)=\begin{cases}
1&{\rm if}\ \sigma_x(j)=i\\
0&{\rm otherwise}
\end{cases}
$$

We conclude that $\sigma_x$ is the permutation in the statement. Summarizing, we have shown that $\varphi$ comes by transposing the map $x\to\sigma_x$, with $\sigma_x$ being as in the statement. By using now the general theory in \cite{bbi}, the Hopf image of $\varphi$ is the algebra $C(G_L)$, with $G_L=<\sigma_1,\ldots,\sigma_K>$, and this finishes the proof.
\end{proof}

Let us discuss the construction of the universal model space, and then Conjecture 4.6, in the classical case. We agree to identify the rank one projections in $M_K(\mathbb C)$ with the corresponding elements of the projective space $P^{K-1}_\mathbb C$.

With these conventions, we first have the following result:

\begin{proposition}
Assuming that $G\subset S_N$ is quasi-transitive, with orbits of size $K$, the universal quasi-flat model space for $G$ is given by $X_G=E_K\times L_{N,K}^G$, where:
$$E_K=\left\{(P_1,\ldots,P_K)\in (P^{K-1}_\mathbb C)^K\Big|P_i\perp P_j,\forall i\neq j\right\}$$
$$L_{N,K}^G=\left\{L\in M_N(*,1,\ldots,K)
\text{ sparse Latin square }
\Big|G_L\subset G\right\}$$ 
In particular, $X_G$ has a canonical probability measure, obtained as the homogeneous space measure on $E_K$ times the the normalized counting measure on $L_{N,K}^G$.
\end{proposition}

\begin{proof}
The first assertion follows by combining Proposition 5.2 and Proposition 5.3 above, and the second assertion is clear from the definitions.
\end{proof}

Note that the model above is empty if there is no sparse Latin square $L$ such that $G_L\subset G$. The existence of such a sparse Latin square is a strong condition on $G$, and here is an intrinsic characterization of the groups satisfying this condition:

\begin{proposition}
Let $G\subset S_N$ be a classical group. The following are equivalent:
\begin{enumerate}
\item $L_{N,K}^G\neq \emptyset$.

\item There exist $K$ elements $\sigma_1, \dots, \sigma_K\in G$ such that $\sigma_1(i), \dots, \sigma_K(i)$ are pairwise distinct for all $1\leq i\leq N$.
\end{enumerate}
\end{proposition}

\begin{proof}
$(1)\implies (2)$ If $\sigma_1, \dots, \sigma_K$ are the permutations associated to $L$, then $\sigma_x(i) = \sigma_y(i)$ means by definition that $x = L_{i\sigma_x(i)} = L_{i\sigma_y(i)} = y$.

$(2)\implies (1)$ Consider such permutations $\sigma_1,\ldots,\sigma_K$. If $i$ is fixed then for each $j$ there is at most one index $x$ such that $\sigma_x(i) = j$. We set $L_{ij} = x$ in that case and $L_{ij} = *$ otherwise. Then, $L$ is a sparse Latin square and the associated permutations are $\sigma_1^{-1},\ldots,\sigma_K^{-1}$, which belong to $G$.
\end{proof}

It turns out that as soon as $L_{N,K}^G\neq \emptyset$, the universal quasi-flat model is stationary. To prove this, let us first give another basic observation:

\begin{proposition}
Given $G\subset S_N$, we have an action $G\curvearrowright L_{N,K}^G$, given by
$$(L^\tau)_{ij}=L_{\tau^{-1}(i)j}$$
\end{proposition}

\begin{proof}
Given a sparse Latin square $L\in M_N(*,1,\ldots,K)$, the matrix $L^\tau\in M_N(*,1,\ldots,K)$ constructed in the statement is a sparse Latin square too and $\tau\to (L\mapsto L^\tau)$ is a group morphism, so it remains to check that $G_L\subset G$ implies $G_{L^\tau}\subset G$.

To do this, let us write $G_L=<\sigma_1,\ldots,\sigma_K>$ and $G_{L^\tau}=<\sigma_1',\ldots,\sigma_K'>$. Then:
\begin{eqnarray*}
\sigma'_x(j)=i
&\iff&(L^\tau)_{ij}=x\iff L_{\tau^{-1}(i)j}=x\\
&\iff&\sigma_x(j)=\tau^{-1}(i)\iff\tau\sigma_x(j)=i
\end{eqnarray*}

Thus, $\sigma_x'=\tau\sigma_x$ and:
\begin{eqnarray*}
G_L\subset G
&\iff&\sigma_1,\ldots,\sigma_K\in G\\
&\iff&\tau\sigma_1,\ldots,\tau\sigma_K\in G\\
&\iff&\sigma_1',\ldots,\sigma_K'\in G\\
&\iff&G_{L^\tau}\subset G
\end{eqnarray*}

Thus $G_L\subset G$ implies $G_{L^\tau}\subset G$, and this finishes the proof.
\end{proof}

We can now verify Conjecture 4.6 in the classical group case:

\begin{theorem}
If the space $L_{N,K}^G$ is not empty, then the universal flat model for a quasi-transitive subgroup $G\subset S_N$,
$$\pi:C(G)\to M_K(C(X_G))$$
is stationary with respect to $\nu\otimes m$ where $m$ is the normalized counting measure on $L_{N,K}^G$ and $\nu$ is any probability measure on $E_K$.
\end{theorem}

\begin{proof}
Write $M_K(C(X_G)) = M_{K}(\mathbb C)\otimes C(E_K)\otimes C(L_{N,K}^G)$ and consider, for $P\in E_K$, the following map:
$$\pi^P = (id\otimes ev_P\otimes id) : C(G)\to M_K(C(L_{N,K}^G))$$

Recall that for $\tau\in G$ and $f\in C(G)$, $\tau.f$ denotes the map $h\mapsto f(\tau^{-1}h)$. This corresponds to the regular action of $G$ on itself. Moreover:
$$\pi^P(u_{ij})(\tau^{-1}L) = \pi^P(u_{\tau(i)j}) = \pi^P(\tau.u_{ij})$$

Since the normalized counting measure $m$ on $L_{N,K}^G$ is $G$-invariant, it follows (writing again $m$ for the integration with respect to $m$ on $C(L_{N, K}^G)$) that $(tr\otimes m)\pi^P$ is a $G$-invariant state on $C(G)$, i.e. it is $\int_G$. Summarizing, we have proved that $(tr\otimes id\otimes m)\pi$ is the constant function equal to $\int_G$ and the result follows.
\end{proof}

\section{Group duals}

In what follows we discuss universal quasi-flat models in the group dual case. Let us start with $\Gamma=\mathbb Z_N$, where the class of $1$ will be denoted by $g$, and called the canonical generator. We have:

\begin{proposition}
We have an isomorphism of Hopf algebras
$$C^*(\mathbb Z_N)=C(S_N^+)\Big/\left<u_{ij}=u_{kl}\Big|j-i=l-k\ ({\rm mod}\ N)\right>$$
which is such that $g=\sum_{j=1}^Nw^{j-i}u_{ij}$, for any $i$, where $w=e^{2\pi i/N}$.
\end{proposition}

\begin{proof}
The quotient algebra $A$ in the statement being generated by the entries of the first row of $u=(u_{ij})$, it is commutative. If we identify the elements of $\widehat{\mathbb Z}_N$ with the powers of $w$, and the elements of $\mathbb Z_N$ with the functions $w\mapsto w^k$, then $\pi(u_{ij})=\delta_{w^{j-i}}$ defines an isomorphism between $A$ and the algebra $C^*(\mathbb Z_N) = C(\widehat{\mathbb Z}_N)$. Moreover:
$$(\pi\otimes \pi)\Delta(u_{ij})=\sum_k\delta_{w^{k-i}}\otimes \delta_{w^{j-k}} = \Delta_{\widehat{\mathbb Z}_K}(\delta_{w^{j-i}})$$

Thus we have an isomorphism of Hopf algebras, as stated. The formula for $g$ in the statement then follows from the formula for $\pi$.
\end{proof}

This translates into a convenient description of the universal flat matrix model.

\begin{proposition}
The universal flat model space for $C^*(\mathbb Z_N)$ is the space
$$E_N=\left\{(P_1,\ldots,P_N)\in (P^{N-1}_\mathbb C)^N\Big|P_i\perp P_j,\forall i\neq j\right\}$$
with the model map given by $\pi(g)(P_1,\ldots,P_N)=\sum_{j=1}^Nw^j P_j$.
\end{proposition}

\begin{proof}
According to Proposition 6.1 above, $\pi(u_{ij})(P_1,\ldots,P_N)=P_{j-i}$ is a flat matrix model map, which gives the formula in the statement for $\pi(g)$.

Let $X$ be a flat matrix model space for $C^*(\mathbb Z_N)$, with model map $\pi_X$. For each $x\in X$ the matrix $\pi_X(g)(x)$ is circulant, so this matrix is completely determined by its first row, which is an element of $E_N$. Let us set $P_j^x=P_{Nj}^x$, and let $P^x$ be the circulant matrix with first row $(P_1^x,\ldots,P_N^x)$. The map $\Phi:X\to E_N$ given by $\Phi(x)=P^x$ being a continuous embedding, we get a surjective $*$-homomorphism $\Psi:M_N(C(E_N))\to M_N(C(X))$ through which $\pi_X$ factors by construction. Thus, we obtain the universality of $E_N$, as claimed.
\end{proof}

Regarding now the general group dual case, we have here the following result:

\begin{theorem}
Given a quotient group $\mathbb Z_K^{* M}\to\Gamma$, the universal quasi-flat model $\pi:C^*(\Gamma)\to M_K(C(Y_\Gamma))$ appears as follows:
\begin{enumerate}

\item For $\Gamma=\mathbb Z_K^{* M}$, the model space is $Y_{\Gamma} = Y_{N,K}=\underbrace{E_K\times\ldots\times E_K}_{M\ terms}$, and the model map is given by the formula in Proposition 6.2, on each of the components.

\item In general, the model space is the subspace $Y_{\Gamma}\subset Y_{N,K}$ consisting of the elements $x$ such that $\pi(\gamma)(x)=Id$ for any $\gamma\in\ker(\mathbb Z_K^{*M}\to\Gamma)$.
\end{enumerate}
\end{theorem}

\begin{proof}
The first assertion is clear from Proposition 6.2. Regarding now the second assertion, let us go back to Theorem 3.5 (2) above. The statement there tells us that the model space for $\widehat{\Gamma}$ appears from the model space for $\underbrace{S_K^+\,\hat{*}\,\ldots\,\hat{*}\,S_K^+}_{M\ terms}$ via the Tannakian conditions defining $\widehat{\Gamma}$. But these Tannakian conditions, when expressed in terms of the generators of $\mathbb Z_K^{* M}$ are precisely the relations $\gamma=1$ with $\gamma\in\ker(\mathbb Z_K^{*M}\to\Gamma)$, as stated.
\end{proof}

\section{Maximal tori}

We will now prove Conjecture 4.7 for several families of duals of discrete groups. The proofs will be more convenient to write by using a lift of the quasi-flat models in the following sense:

\begin{proposition}
The affine lift of the universal quasi-flat model for $C^*(\mathbb Z_K^{\ast M})$ is $\pi:C^*(\mathbb Z_K^{\ast M})\to M_K(C(U_K^M))$ given on the canonical generator $g_{i}$ of the $i$-th factor by
$$\pi(g_{i})(U^{1}, \dots, U^{M})=\sum_jw^jP_{U_j^i}$$
where $U_j^i$ is the $j$-th column of $U^i$ and $P_{\xi}$ denotes the orthogonal projection onto $\mathbb C\xi$.
\end{proposition}

\begin{proof}
There is indeed a canonical quotient map $U_K\to E_K$, obtained by parametrizing the orthonormal bases of $\mathbb C^K$ by the unitary group $U_K$, and this gives the result.
\end{proof}

According to the results of \cite{bic} explained in section 1 above, the maximal group dual subgroups $\widehat{\Gamma}\subset S_N^+$, which can be regarded as being ``maximal tori'', are the free products of type $\mathbb Z_{K_1}*\ldots*\mathbb Z_{K_M}$ with $K_1+\ldots+K_M=N$. In the quasi-transitive case, where $K_1=\ldots=K_M=K$ with $K|N$, we have the following result:

\begin{theorem}
The universal quasi-flat model for $\mathbb Z_K^{*M}$ is inner faithful.
\end{theorem}

\begin{proof}
It is enough to prove that the affine lift of the model is inner faithful. Let us consider a reduced word $\gamma\in\mathbb Z_K^{*M}$, and write it as $\gamma=g_{i_1}^{k_1}\ldots g_{i_n}^{k_n}$, with $i_t\neq i_{t+1}$, and with $1\leq k_t\leq K-1$. Then:
\begin{eqnarray*}
\pi(\gamma)(U^1,\ldots ,U^M) 
&=&\sum_{j_1\ldots j_n=1}^Kw^{k_1j_1}P_{U^{i_1}_{j_1}}\ldots w^{k_nj_n}P_{U^{i_n}_{j_n}}\\
&=&\sum_{j_1\ldots j_n=1}^Kw^{<k,j>}P_{U^{i_1}_{j_1}}\ldots P_{U^{i_n}_{j_n}}
\end{eqnarray*}

Our aim is to prove that there is at least one tuple $(U^1,\ldots ,U^M)$ for which the matrix above is not the identity. Recall the following formula valid for any vectors in a Hilbert space $\xi_1,\ldots,\xi_l\in H$, with the scalar product being linear on the left:
$$P_{\xi_1}\ldots P_{\xi_l}(x)=<x,\xi_l><\xi_l,\xi_{l-1}>\ldots\ldots<\xi_2, \xi_1>\xi_1$$

To compute the trace of this operator, one can consider any orthonormal basis containing $\xi_l$, yielding $<\xi_1,\xi_l><\xi_l,\xi_{l-1}>\ldots\ldots<\xi_2, \xi_1>$. Applying this to $\pi(\gamma)$ and using the equality $<U_i,V_j>=\sum_lU_{ki}\bar{V}_{kj}=(V^*U)_{ji}$, we get:
\begin{eqnarray*}
tr\circ\pi(\gamma)
&=&\frac{1}{K}\sum_{j_1\ldots j_n=1}^Kw^{<k,j>}<U^{i_1}_{j_1}, U^{i_n}_{j_n}><U^{i_n}_{j_n},U^{i_{n-1}}_{j_{n-1}}>\ldots <U^{i_2}_{j_2},U^{i_1}_{j_1}>\\
&=&\frac{1}{K}\sum_{j_1\ldots j_n=1}^Kw^{<k,j>}(U^{i_n*}U^{i_1})_{j_nj_1}(U^{i_{n-1}*}U^{i_n})_{j_{n-1}j_n}\ldots (U^{i_1*}U^{i_2})_{j_1j_2}
\end{eqnarray*}

Denoting by $W$ the diagonal matrix given by $W_{ij}=\delta_{ij}w^i$, we have:
$$\sum_{j_1}w^{k_1j_1}U_{j_nj_{1}}U^*_{j_1j_2}=\sum_{j_1l}U_{j_nj_1}W^{k_1}_{j_1l}U^*_{lj_2}=(UW^{k_1}U^*)_{j_nj_2}$$

Applying this $n$ times in the above formula for $tr\circ\pi(\gamma)$ yields:
\begin{eqnarray*}
tr\circ\pi(\gamma)
&=&tr\left(U^{i_n*}U^{i_1}W^{k_1}U^{i_1*}U^{i_2}W^{k_2}\ldots W^{k_{n-1}}U^{i_{n-1}*}U^{i_n}W^{k_n}\right)\\
&=&tr\left(U^{i_1}W^{k_1}U^{i_1*}\ldots U^{i_n}W^{k_n}U^{i_n*}\right)
\end{eqnarray*}

Assume now that $\pi(\gamma)(U^1,\ldots,U^M) = Id$ for all tuples of unitary matrices. The trace of a unitary matrix can only be equal to $1$ if it is the identity, hence:
$$\prod_{p=1}^nU^{i_p}W^{k_p}U^{i_p*} = Id$$

In other words, the following noncommutative polynomial vanishes on $U_K^M$:
$$P = \prod_{p=1}^nX^{i_p}W^{k_p}X^{i_p*} - 1$$

But this is impossible if $k_t\neq 0 \mod K$ for all $t$, hence $\pi(\gamma)$ is not always the identity and $\pi$ is inner faithful.
\end{proof}

The above result can be extended by allowing arbitrary direct products as components of the free product. More precisely, assume that $M = M_1+\ldots +M_n$. Then, the free product $\mathbb Z_{K}^{M_1}*\ldots*\mathbb Z_{K}^{M_n}$ is a quotient of $\mathbb Z_K^{*M}$ which is still quasi-transitive with orbits of size $K$ and we prove that its universal quasi-flat matrix model is still inner faithful.

\begin{proposition}
The universal quasi-flat model for $\Gamma = \mathbb Z_{K}^{M_1}*\ldots*\mathbb Z_{K}^{M_n}$ is inner faithful.
\end{proposition}

\begin{proof}
We first need a picture of the affine lift of the model for the direct product $\mathbb{Z}_K^M$. Note that if two complete families of rank one orthogonal projections commute, then they are a permutation of one another. We may therefore consider the space $U_K\times S_K^M$ as the affine lift of our model. If $g_1, \ldots, g_M$ are the canonical generators of the direct product, their action is then given by:
$$\pi(g_i)(U, \sigma_1, \ldots, \sigma_M) = \sum_{j=1}^Kw^jP_{U_{\sigma_i^{-1}(j)}} = \sum_{j=1}^Kw^{\sigma_i(j)}P_{U_j}$$

This yields the following formula for a general element:
$$\pi(g_1^{k_1}\ldots g_M^{k_M})(U, \sigma_1, \ldots, \sigma_M) = \sum_{j=1}^Kw^{k_1\sigma_1(j) + \ldots + k_M\sigma_M(j)}P_{U_j}$$

Let $g_1(i), \ldots, g_{M_i}(i)$ be the generators of $\mathbb{Z}_K^{M_i}$, and let consider a reduced word:
$$\gamma = \left(g_1(i_1)^{k_1(1)}\ldots g_{M_{i_1}}(i_1)^{k_{M_{i_1}}(1)}\right)\ldots \left(g_1(i_n)^{k_1(n)}\ldots g_{M_{i_n}}(i_n)^{k_{M_{i_n}}(n)}\right)$$
The computation of $tr\circ\pi(\gamma)$ is similar to the one in the proof of Theorem 7.2 until the introduction of the matrices $W$. Here we have to replace $W^{k_t}$ by $\prod_{s=1}^{M_t}W_{\sigma_s^t}^{k_s(t)}$, where $(U^t, \sigma^t_1, \ldots, \sigma^t_{M_t})_{1 \leq t\leq n}$ is the element to which we are applying $\pi(\gamma)$ and $(W_{\sigma})_{ij} = \delta_{ij}w^{\sigma(i)}$.

Assuming that $\pi(\gamma) = 1$, we can apply the same strategy as before: we have a polynomial which must vanish on all tuples of unitary matrices and this is impossible unless all the matrices appearing in the polynomial are the identity. We therefore get the condition $\prod_{s=1}^{M_t}W_{\sigma_s^t}^{k_s(t)} = Id$ for all $t$, which translates into:
$$\prod_{s=1}^{M_t}w^{\sigma_s^t(i)k_s(t)} = 1$$

To derive a contradiction, we sum the above equation over all permutations, getting:
$$\frac{1}{(K!)^M}\sum_{\sigma_1,\ldots,\sigma_{M_i}\in S_K}w^{k_1(t)\sigma_1(i) + \ldots + k_{M_t}(t)\sigma_{M_t}(i)} = \frac{1}{(K!)^M}\prod_{s=1}^{M_t}\left(\sum_{\sigma_s\in S_K}w^{k_s(t)\sigma_s(i)}\right)$$

For any $i'$, there are $(K-1)!$ permutations $\sigma$ such that $\sigma_s(i) = i'$. This leads to:
\begin{eqnarray*}
\sum_{\sigma_s\in S_K}w^{k_s(t)\sigma_s(i)} & = & \sum_{i'=1}^K(K-1)!w^{k_s(t)i'} \\
& = & K!\delta_{k_s(t), 0}
\end{eqnarray*}

Gathering everything yields:
$$\prod_{s=1}^{M_i}\delta_{k_s(t), 0} = 1$$

Thus $k_s(t) = 0$ for all $t$ and all $s$, a contradiction.
\end{proof}

\section{Further results}

In this section we give two other extensions of Theorem 7.2. Before that, let us give an example showing that the universal model is not always inner faithful for group duals:

\begin{proposition}
The group $\Gamma = (\mathbb Z_K * \mathbb Z_K)\times \mathbb Z_K$ has no inner faithful quasi-flat model.
\end{proposition}

\begin{proof}
Let $\pi : C^*(\Gamma)\to M_K(C(X))$ be a matrix model. For all $x$, $\pi(g_1)(x)$ and $\pi(g_3)(x)$ are commuting diagonalizable operators with all eigenvalues of multiplicity one. The same holds for $\pi(g_2)(x)$ and $\pi(g_3)(x)$ so that $\pi(g_1)(x)$ and $\pi(g_2)(x)$ must commute. Since this holds for any $x\in X$, we conclude that $\pi$ factors through the quotient $\mathbb Z_K^3$, so that the model is not inner faithful.
\end{proof}

We extend now Theorem 7.2 by allowing the free product to be amalgamated over a subgroup isomorphic to $\mathbb Z_L$. Assume that we have $N=KM$ as before, and assume in addition that we have $K=LR$. We can then consider the following quotient:
$$\Gamma=\mathbb Z_K^{*M}\Big\slash\left<g_i^R= g_j^R\Big|\forall i,j\right>$$

It is enough to provide an explicit model which is inner faithful and we define it through its affine lift. Let $(e_i)_{1\leq i\leq N}$ be the canonical basis of $\mathbb C^K$ and consider for $1\leq t\leq L$ the subspace $V_t\subset\mathbb C^K$ spanned by the vectors $e_i$ with $(t-1)R+1\leq i\leq tR$, so that:
$$\mathbb C^K=V_1\oplus\ldots\oplus V_L$$

This gives a block-diagonal embedding $U_R^L\subset U_K$, which is the affine lift of the model space. However, for computations it is simpler to use a permutation of this model:

\begin{definition}
We denote by $X_{K,L}$ the set of unitary matrices $U\subset U_K$ such that $(U_{t},U_{t+L},\ldots,U_{t+(R-1)L})$ is a basis of $V_t$ for all $1\leq t\leq L$.
\end{definition}

Defining $\pi(g_i)$ as usual, we see that the following element does not depend on $i$:
$$\pi(g_i)^R(U^1, \ldots, U^M) = \sum_{j=1}^Kw^{jR}P_{U^i_j} = \sum_{t=1}^{L}w^{tR}\left(\sum_{s=0}^{R-1}P_{U^i_{t+sL}}\right) = \sum_{t=1}^L w^{tR}P_{V_t}$$
Thus, we get a representation $\pi : C^*(\Gamma)\to M_N(C(X_{K,L}^M))$. We have:

\begin{theorem}
The quasi-flat model $\pi : C^*(\Gamma)\to M_K(C(X_{K,L}^M))$ for the amalgamated free product $\Gamma=\mathbb Z_K^{*M}\big\slash\left<g_i^R= g_j^R\big|\forall i,j\right>$ is inner faithful.
\end{theorem}

\begin{proof}
As usual, we consider an arbitrary element $\gamma\in \Gamma$. Writing $h = g_i^R$, we can assume that $\gamma = h^lg_{i_1}^{k_1}\ldots g_{i_n}^{k_n}$ with $0<l<L$ and $0< k_i< R$ for all $i$. Thus:
\begin{eqnarray*}
\pi(\gamma)(U^1, \ldots, U^M) & = & \pi(h^l)(U^1, \ldots, U^M)\left(\sum_{j_1\ldots j_n}w^{<k,j>}P_{U^{i_1}_{j_1}}\ldots P_{U^{i_n}_{j_n}}\right)\\
& = &\left(\sum_{t=1}^Lw^{tlR}P_{V_t}\right)\left(\sum_{j_1\ldots j_n}w^{<k,j>}P_{U^{i_1}_{j_1}}\ldots P_{U^{i_n}_{j_n}}\right)\\
& = & \sum_{t=1}^L\sum_{j_1\ldots j_n}w^{tlR}w^{<k,j>}P_{V_t}P_{U^{i_1}_{j_1}}\ldots P_{U^{i_n}_{j_n}}
\end{eqnarray*}

Let us consider the term corresponding to a fixed $t$. The product of projections vanishes unless $U_{j_1}\in V_t$. This forces $U_{i_2}\in V_{t}$ and by induction all the terms must be in $V_t$. This means that there are $s_1, \ldots, s_n$ such that $j_m = t + s_mL$ for all $1\leq m\leq n$. Thus:
$$\pi(\gamma)(U^1,\ldots ,U^M) = \sum_{t=1}^L w^{tlR}w^{t(k_1 + \ldots + k_n)}\sum_{s_1,\ldots, s_n = 0}^{R-1}w^{L<k,s>}P_{U^{i_1}_{t+s_1L}}\ldots P_{U^{i_n}_{t+s_nL}}$$

Let us denote by $U^i(t)$ the unitary operator on $V_t$ obtained from the appropriate columns of $U^i$. The trace of $\pi(\gamma)$ can be expressed using these operators. To do this, simply note that the normalized trace $tr$ can be written as $L^{-1}tr_R$, where $tr_R$ is the normalized trace on $U_R$. A computation similar to that of Theorem 7.2 then yields:
$$tr\circ\pi(\gamma) = \frac{1}{L}\sum_{t=1}^L w^{t(lR + k_1 + \ldots + k_n)}tr_R\left(\prod_{p=1}^n(U^{i_p}(t)W^{Lk_p}U^{i_p*}(t))\right)$$

Assuming that $\pi(\gamma) = Id$, we can now derive a contradiction. Indeed, this forces:
$$\frac{1}{L}\sum_{t=1}^L w^{t(lR + k_1 + \ldots + k_n)}\prod_{p=1}^n(U^{i_p}(t)W^{Lk_p}U^{i_p}(t)) = Id$$

Because $U^i(t)$ only acts on $V_t$ and the decomposition is orthogonal, this equation is equivalent to the system formed by the equations for each fixed $t$. As before, this system cannot always be satisfied unless $k_p=0$ for all $p$. In that case, we are left with $w^{tlR} = 1$ for all $t$, implying $l=0$ and the proof is complete.
\end{proof}

We end with another construction. This time, we do not identify the copies of $\mathbb Z_L$ but simply make them commute. More precisely, we set:
$$\Gamma = \mathbb{Z}_{K}^{\ast M}\big\slash\left<g_i^Rg_j^R = g_j^Rg_i^R \big| \forall i, j\right>$$

This is in a sense a mix between the free product of direct products and the amalgamated free products. Thus, the model space is $X_\Gamma = X_{K, L}^M\times S_{L}^M$ understood as tuples $(U^1, \ldots, U^M, \sigma_1, \ldots, \sigma_M)$ where $U^i$ satisfies that $(U_{t}^i, U_{t + L}^i, \ldots, U_{t+(r-1)L}^i)$ is a basis of $V_{\sigma_i^{-1}(t)}$ for all $t$. Still using the usual formula for the canonical generators, we have:
$$\pi(g_i)^R(U^1, \ldots, U^M, \sigma_1, \ldots, \sigma_M) = \sum_{t=1}^{L} w^{\sigma_i(t)R}P_{V_t}$$

This element commutes with $\pi(g_j)^R$ so that we indeed have a universal representation $\pi : C^*(\Gamma)\rightarrow M_K(C(X_\Gamma))$. With this convention, we have:

\begin{proposition}
The quasi-flat model $\pi : C^*(\Gamma)\to M_K(C(X_\Gamma))$ for the group $\Gamma = \mathbb{Z}_{K}^{\ast M}\big\slash\left<g_i^Rg_j^R = g_j^Rg_i^R \big| \forall i, j\right>$ is inner faithful.
\end{proposition}

\begin{proof}
As usual, we consider a non-trivial word $\gamma=g_{i_1}^{k_1}\ldots g_{i_n}^{k_n}$ in a particular form. For each $i$, let $k_i = k'_i+a_iR$ be the euclidean division of $k_i$ by $R$. Then:
\begin{eqnarray*}
\pi(\gamma)(U^1, \ldots, U^M, \sigma_1, \ldots, \sigma_M) 
& = & \prod_{s=1}^n\left(\sum_{t_s=1}^{L} w^{\sigma_{i_s}(t_s)a_sR}P_{V_{t_s}}\right)\left(\sum_{j_s=1}^K  w^{k_s'j_s}P_{U^{i_s}_{j_s}}\right) \\
& = & \sum_{t_1, \ldots, t_n=1}^{L}\sum_{j_1, \ldots, j_s=1}^Kw^{R<\sigma(t),a>}w^{<k',j>}P_{V_{t_1}}P_{U_{j_1}}^{i_1}\ldots P_{V_{t_n}}P_{U_{j_n}}^{i_n}
\end{eqnarray*}

For the product to be nonzero we need all the $V_t$'s to be the same so that the first sum reduces to only one index. Moreover, this forces as before $j_m = t + s_mL$, so we get:
$$\sum_{t=1}^{L}\sum_{s_1,\ldots, s_n = 0}^{R-1}w^{R(a_1\sigma_{i_1}(t) + \ldots + a_n\sigma_{i_n}(t))}w^{t(k'_1+\ldots+k'_n)}w^{L<k',s>}P_{U^{i_1}_{t+s_1L}}\ldots P_{U^{i_n}_{t+s_nL}}$$

As in the proof of Theorem 8.1, we conclude that $k'_p=0$ for all $p$ and that:
$$w^{R(a_1\sigma_{i_1}(t) + \ldots + a_n\sigma_{i_n}(t))} = 1$$

Summing over $S_L^M$ then yields $\prod\delta_{a_i, 0} = 1$ and the proof is complete.
\end{proof}


\begin{thebibliography}{99}

\bibitem{ban}T. Banica, Quantum groups from stationary matrix models, {\em Colloq. Math.} {\bf 148} (2017), 247--267.

\bibitem{bbi}T. Banica and J. Bichon, Hopf images and inner faithful representations, {\em Glasg. Math. J.} {\bf 52} (2010), 677--703.

\bibitem{bbs}T. Banica, J. Bichon and J.-M. Schlenker, Representations of quantum permutation algebras, {\em J. Funct. Anal.} {\bf 257} (2009), 2864--2910.

\bibitem{bco}T. Banica and B. Collins, Integration over quantum permutation groups, {\em J. Funct. Anal.} {\bf 242} (2007), 641--657.

\bibitem{bfs}T. Banica, U. Franz and A. Skalski, Idempotent states and the inner linearity property, {\em Bull. Pol. Acad. Sci. Math.} {\bf 60} (2012), 123--132.

\bibitem{bne}T. Banica and I. Nechita, Flat matrix models for quantum permutation groups, {\em Adv. Appl. Math.} {\bf 83} (2017), 24--46.

\bibitem{bic}J. Bichon, Algebraic quantum permutation groups, {\em Asian-Eur. J. Math.} {\bf 1} (2008), 1--13.

\bibitem{bcv}M. Brannan, B. Collins and R. Vergnioux, The Connes embedding property for quantum group von Neumann algebras,  {\em Trans. Amer. Math. Soc.} {\bf 369} (2017), 3799--3819.

\bibitem{cha}A. Chassaniol, Quantum automorphism group of the lexicographic product of finite regular graphs, {\em J. Algebra} {\bf 456} (2016), 23--45. 

\bibitem{chi}A. Chirvasitu, Residually finite quantum group algebras, {\em J. Funct. Anal.} {\bf 268} (2015), 3508--3533.

\bibitem{cdp}L.S. Cirio, A. D'Andrea, C. Pinzari and S. Rossi, Connected components of compact matrix quantum groups and finiteness conditions, {\em J. Funct. Anal.} {\bf 267} (2014), 3154--3204.

\bibitem{fsk}U. Franz and A. Skalski, On idempotent states on quantum groups, {\em J. Algebra} {\bf 322} (2009), 1774--1802.

\bibitem{fre}A. Freslon, On the partition approach to Schur-Weyl duality and free quantum groups, {\em Transform. Groups} {\bf 22} (2017), 707--751.

\bibitem{mal}S. Malacarne, Woronowicz's Tannaka-Krein duality and free orthogonal quantum groups, {\em Math. Scand.} {\bf 122} (2018), 151--160.

\bibitem{ntu}S. Neshveyev and L. Tuset, Compact quantum groups and their representation categories, SMF (2013).

\bibitem{rwe}S. Raum and M. Weber, The full classification of orthogonal easy quantum groups, {\em Comm. Math. Phys.} {\bf 341} (2016), 751--779.

\bibitem{sin}R. Sinkhorn, A relationship between arbitrary positive matrices and doubly stochastic matrices, {\em Ann. Math. Statist.} {\bf 35} (1964), 876--879. 

\bibitem{tho}E. Thoma, \"Uber unit\"are Darstellungen abz\"ahlbarer, diskreter Gruppen, {\em Math. Ann.} {\bf 153} (1964), 111--138.

\bibitem{wa1}S. Wang, Free products of compact quantum groups, {\em Comm. Math. Phys.} {\bf 167} (1995), 671--692.

\bibitem{wa2}S. Wang, Quantum symmetry groups of finite spaces, {\em Comm. Math. Phys.} {\bf 195} (1998), 195--211.

\bibitem{wa3}S. Wang, Equivalent notions of normal quantum subgroups, compact quantum groups with properties F and FD, and other applications, {\em J. Algebra} {\bf 397} (2014), 515--534.

\bibitem{wa4}S. Wang, $L_p$-improving convolution operators on finite quantum groups, {\em Indiana Univ. Math. J.} {\bf 65} (2016), 1609--1637.

\bibitem{wo1}S.L. Woronowicz, Compact matrix pseudogroups, {\em Comm. Math. Phys.} {\bf 111} (1987), 613--665.

\bibitem{wo2}S.L. Woronowicz, Tannaka-Krein duality for compact matrix pseudogroups. Twisted SU(N) groups, {\em Invent. Math.} {\bf 93} (1988), 35--76.

\end{thebibliography}
\end{document}